\documentclass[12pt]{amsart}

\usepackage{amsmath, amssymb}
\usepackage{color}

\newtheorem{theorem}{Theorem}[section]
\newtheorem{lemma}[theorem]{Lemma}
\newtheorem{prop}[theorem]{Proposition}

\theoremstyle{definition}

\newtheorem{example}[theorem]{Example}

\theoremstyle{remark}
\newtheorem{remark}[theorem]{Remark}

\numberwithin{equation}{section}

\newcommand{\rr}{{\mathbb R}}
\newcommand{\CC}{\mathbb{C}}
\newcommand{\RR}{\mathbb{R}}
\newcommand{\rd}{{\mathbb R^d}}

\newcommand{\Ker}{\operatorname{Ker}}
\newcommand{\Spec}{\operatorname{Spec}}
\newcommand{\RSpec}{\operatorname{ReSpec}}
\renewcommand{\Re}{\operatorname{Re}}

\newcommand{\dd}{\mathrm{d}}
\newcommand{\ee}{\mathrm{e}}
\newcommand{\cI}{{\mathcal I}}
\newcommand{\cF}{{\mathcal F}}

\newcommand{\EE}{\operatorname{\mathbb{E}}}

\allowdisplaybreaks

\begin{document}
\sloppy
\title[Operator scaled Wiener bridges]{Operator scaled Wiener bridges}

\author{M\'aty\'as Barczy}
\address{M\'aty\'as Barczy: Faculty of Informatics, University of Debrecen, Pf.12, H--4010 Debrecen, Hungary}
\email{barczy.matyas\@@{}inf.unideb.hu}
\thanks{The research of M. Barczy was realized in the frames of
 T\'AMOP 4.2.4.\ A/2-11-1-2012-0001 ,,National Excellence Program --
 Elaborating and operating an inland student and researcher personal support
 system''.
The project was subsidized by the European Union and co-financed by the
 European Social Fund.}

\author{Peter Kern}
\address{Peter Kern, Mathematisches Institut, Heinrich-Heine-Universit\"at D\"usseldorf, Universit\"atsstr. 1, D--40225 D\"usseldorf, Germany}
\email{kern\@@{}math.uni-duesseldorf.de}

\author{Vincent Krause}
\address{Vincent Krause, Keltenstr. 19, D--41462 Neuss, Germany}
\email{vincentkrause\@@{}arcor.de}

\date{\today}

\begin{abstract}
We introduce operator scaled Wiener bridges by incorporating a matrix scaling in the drift part of
 the SDE of a multidimensional Wiener bridge.
A sufficient condition for the bridge property of the SDE solution is derived in terms of the eigenvalues of the scaling matrix.
We analyze the asymptotic behavior of the bridges and briefly discuss the question whether the scaling matrix determines uniquely
 the law of the corresponding bridge.
\end{abstract}



\maketitle

\baselineskip=18pt

\section{Introduction}

This paper deals with a multidimensional generalization of the so-called $\alpha$-Wiener bridges also known as scaled Wiener bridges.
For fixed $T>0$ and given matrices $A\in\rr^{d\times d}$ and $\Sigma\in\rr^{d\times m}$,
 a $d$-dimensional process $(X_t)_{t\in[0,T)}$ is given by the SDE
\begin{equation}\label{msde}
 \dd X_t=-\,\frac{1}{T-t}AX_t\,\dd t+\Sigma\,\dd B_t, \quad t\in[0,T),
\end{equation}
with initial condition $X_0=0\in\rd$, where $(B_t)_{t\in[0,T)}$ is an $m$-dimensional standard Wiener process defined
 on the filtered probability space $(\Omega,\cF,(\mathcal F_t)_{t\in[0,T)},P)$ with the completion $(\mathcal F_t)_{t\in[0,T)}$ of the canonical filtration of $(B_t)_{t\in[0,T)}$. 
Note that in case $m=d$ and if $A$ and $\Sigma$ are both the $d\times d$ identity matrix, then the process $(X)_{t\in[0,T)}$ is nothing else but the
 usual $d$-dimensional Wiener bridge over $[0,T]$.

To our knowledge, in case of dimension $d=1$, these kinds of processes have been first considered
 by Brennan and Schwartz \cite{BreSch}; see also Mansuy \cite{Man}.
In Brennan and Schwartz \cite{BreSch} $\alpha$-Wiener bridges, where $A=\alpha\in\RR^{1\times 1}$ with $\alpha>0$,
 are used to model the arbitrage profit associated with a given futures contract in the absence
 of transaction costs.
This model is also meaningful in a multidimensional context when a finite number of contracts is considered with possible
 dependencies between the contracts.
Operator scaled Wiener bridges offer a tool for modeling the arbitrage profit in this multidimensional setting.

Sondermann, Trede and Wilfling  \cite{SonTreWil} and Trede and Wilfling \cite{TreWil}
used $\alpha$-Wiener bridges with $\alpha>0$ to describe the fundamental component
 of an exchange rate process and they call the process a scaled Brownian bridge.
The essence of these models is that the coefficient $-\alpha/(T-t)$ of $X_t$ in the drift term in \eqref{msde} represents some kind of mean
 reversion, a stabilizing force that keeps pulling the process towards its mean $0$,
  and the absolute value of this force is increasing proportionally to the inverse of the remaining time $T-t,$
 with the constant rate $\alpha$.
This model is used in \cite{TreWil} to analyze the exchange rate of the Greek drachma to the Euro before the Greek EMU entrance on 1 January 2001 with a priorly fixed conversion rate. Trede and Wilfling \cite{TreWil} observe an increase in interventions towards the fixed conversion rate, well described by an $\alpha$-Wiener bridge plus deterministic drift with MLE-estimator $\widehat\alpha=1.24$. If more than two countries join the EMU at the same time, most recently Cyprus and Malta on 1 January 2008, operator scaled Wiener bridges may offer a useful tool to analyze interventions for all the exchange rates, commonly. In this context the replacement of a constant rate $\alpha$ by some scaling matrix $A$ is meaningful, since the economies of EU countries are tightly linked and thus interventions are likely to be strongly dependent on each other.

The SDE \eqref{msde} with initial condition $X_0=0$ has a unique strong solution $(X_t)_{t\in[0,T)}$ given by the
 $d$-dimensional integral representation
\begin{equation}\label{mintrep}
X_t=\int_0^t\Big(\frac{T-t}{T-s}\Big)^A\Sigma\,\dd B_s\quad\text{ for }t\in[0,T),
\end{equation}
where $r^A$ is defined by the exponential operator
 \begin{align}\label{help1}
 r^A=\ee^{A\log r}=\sum_{k=0}^\infty\frac{(\log r)^k}{k!}\,A^k\quad\text{ for }r>0.
 \end{align}
The validity of \eqref{mintrep} can be easily checked using It\^{o}'s formula and properties of the exponential operator. Indeed,
 \begin{align*}
 \dd X_t & = \left( \left(\frac{\dd}{\dd t} (T-t)^A\right)
                   \int_0^t (T-s)^{-A}\Sigma \,\dd B_s
             \right)\dd t
             + (T-t)^A (T-t)^{-A}\Sigma\,\dd B_t \\
         & = \left( \left(-A(T-t)^{A-I_d} \right)
                   \int_0^t (T-s)^{-A}\Sigma \,\dd B_s
             \right)\dd t
              + \Sigma\,\dd B_t \\
         & = -\frac{1}{T-t}AX_t\,\dd t + \Sigma\,\dd B_t,
             \quad t\in[0,T),
 \end{align*}
 where $I_d$ denotes the $d\times d$ identity matrix.
Further, by Section 5.6 in Karatzas and Shreve \cite{KS}, strong uniqueness holds for the SDE \eqref{msde}.
Note also that $(X_t)_{t\in[0,T)}$ is a Gauss process with almost surely continuous sample paths,
 see, e.g., Problem 5.6.2 in Karatzas and Shreve \cite{KS}.
Later on, we will frequently assume that $\Sigma$ has rank $d$ (and consequently $m\geq d$), but the assumption will always be stated explicitly. Note that this is only a minor restriction, since otherwise the $d$-dimensional Gaussian driving process $(\Sigma B_t)_{t\in[0,T)}$ in \eqref{mintrep} has linearly dependent coordinates.

The paper is organized as follows.
In Section \ref{Section_preliminaries} we recall a spectral decomposition of the matrix $A$ and of the process $X$, respectively.
We further present a result on the growth behavior of the exponential operator $t^A$ near the origin, and
 we also recall a strong law of large numbers and a law of the iterated logarithm valid for the martingale
 $((T-t)^{-A}X_t)_{t\in[0,T)}$.
In Section \ref{Section_bridge_property}, in order to properly speak of a process bridge, we derive
 some sufficient conditions on $A$ and $\Sigma$
 such that $X_t$ converges to the origin almost surely as $t\uparrow T$, see Theorem \ref{Theorem1}.
Provided that the conditions of Theorem \ref{Theorem1} hold, we will call the process $(X_t)_{t\in[0,T]}$ an
{\it operator scaled Wiener bridge} associated to the matrices $A$ and $\Sigma$ over the time interval $[0,T]$.
By giving an example, we point out that if the conditions of
 Theorem \ref{Theorem1} do not hold, then in general one cannot expect that $X_t$ converges to some
 deterministic $d$-dimensional vector almost surely as $t\uparrow T$.
Section \ref{Section_asymptotic} is devoted to study the asymptotic behavior of the sample paths of operator
 scaled Wiener bridges as $t\uparrow T$.
Finally, in Section \ref{Section_uniqueness} we address the question of uniqueness of bridges.
By giving examples, we point out that there exist matrices $A$, $\widetilde A\in\RR^{d\times d}$ and $\Sigma\in\RR^{d\times m}$
 such that the laws of the bridges associated to the matrices $A$ and $\Sigma$, and $\widetilde A$ and $\Sigma$
 coincide, but $A\ne\widetilde A$.
We also formulate a partial result on the uniqueness of bridges in terms of the spectrum of $A$, see Proposition \ref{Prop_uniqueness}.

\section{Preliminaries}\label{Section_preliminaries}

\subsection{Spectral decomposition}\label{sebsection_spectral}

Factor the minimal polynomial $f$ of $A$ into $f(\lambda)=f_{1}(\lambda)\cdots f_{p}(\lambda)$, $\lambda\in\CC$,
 with $p\leq d$ such that every root of $f_{j}$ has real part $a_{j}$, where $a_1<\cdots<a_p$ denote the distinct
 real parts of the eigenvalues of $A$.
Note that $f$, $f_1,\ldots,f_p$ are polynomials with real coefficients.
According to the primary decomposition theorem of linear algebra we can decompose $\rd$ into a direct sum $\rd=V_{1}\oplus\cdots\oplus V_{p}$, where each $V_{j}:=\Ker(f_{j}(A))$ is an $A$-invariant subspace.
Let us denote the dimension of $V_j$ by $d_j$, $j=1,\ldots,p$.
Now, in an appropriate basis, say $\{b_i^{(j)}: i=1,\ldots, d_j,\, j=1,\ldots,p\}$,
 $A$ can be represented as a block-diagonal matrix $A=A_{1}\oplus\cdots\oplus A_{p}$, where
 every eigenvalue of $A_{j}$ has real part $a_{j}$.
For this reason, we will call each matrix $A_j$ {\it real spectrally simple}, i.e., all its eigenvalues have the same real part.
We can further choose a unique inner product $\langle \cdot,\cdot\rangle$ on $\rd$ such that
 the basis $\{b_i^{(j)} : i=1,\ldots, d_j,\, j=1,\ldots,p\}$ is orthonormal, and consequently, the
 subspaces $V_j$, $1\leq j\leq p$, are mutually orthogonal.
For $x=x_1+\cdots+x_p$ with $x_j\in V_j$, $j=1,\ldots,p$, let $\pi_j(x)$ be the coordinates of $x_j$ with respect to
 the basis $\{b_i^{(j)}: i=1,\ldots,d_j\}$ of $V_j$.
Then $\pi_j : \RR^d \to \RR^{d_j}$ is a linear projection mapping.
To conclude, for every $x\in\RR^d$ there exist unique $x_j\in V_j$, $j=1,\ldots,p$, such that $x=x_1+\cdots+x_p=(\pi_1(x),\ldots,\pi_p(x))$
 and $t^Ax=(t^{A_1}\pi_1(x),\ldots,t^{A_p}\pi_p(x))$ for all $t>0$.
This later fact is a consequence of $t^A = t^{A_1}\oplus\cdots\oplus t^{A_p}$ which can be easily checked using \eqref{help1}.
Moreover, for our multidimensional process we have $X_t=(X^{[1]}_t,\ldots,X^{[p]}_t)$,
 where $(X^{[j]}_t =\pi_j(X_t))_{t\in[0,T)}$ is again of the same structure \eqref{mintrep} which will be shown below
 in Lemma \ref{lemma_spectral1}.
Thus it suffices to show that for each component $X^{[j]}_t\to0\in\RR^{d_j}$ almost surely to deduce $X_t\to0\in\RR^d$
 almost surely as $t\uparrow T$.%

\begin{lemma}\label{lemma_spectral1}
For every $j=1,\ldots,p$, the $j$-th spectral component of $(X_t)_{t\in[0,T)}$ can almost surely be represented as
 \begin{align}\label{help4}
  X_t^{[j]}=\int_0^t\Big(\frac{T-t}{T-s}\Big)^{A_j}\Sigma_j\,\dd B_s\quad\text{ for }t\in[0,T),
 \end{align}
 where $\Sigma_j\in\RR^{d_j\times m}$ is given by $\pi_j(\Sigma y) = \Sigma_j y$ for $y\in\RR^m$.
\end{lemma}

\begin{proof}
The mapping $\RR^m\ni y\mapsto \pi_j(\Sigma y)\in\RR^{d_j}$ is linear for every $j=1,\ldots,p$ and hence there exists a matrix
 $\Sigma_j\in\RR^{d_j\times m}$ such that $\pi_j(\Sigma y)=\Sigma_j y$ for $y\in\RR^m$.
Then
 \begin{align}\label{help6}
    t^A\Sigma y = (t^{A_1}\Sigma_1 y, \ldots, t^{A_p}\Sigma_p y) \qquad
       \text{for all $y\in\RR^m$ and $t>0$,}
 \end{align}
 and hence almost surely
 \begin{align*}
   X_t^{[j]} & = \pi_j(X_t) = \pi_j\left( \int_0^t\Big(\frac{T-t}{T-s}\Big)^A\Sigma\,\dd B_s \right)\\
             & = \pi_j \left( \int_0^t\Big(\frac{T-t}{T-s}\Big)^{A_1}\Sigma_1\,\dd B_s,\ldots,
                            \int_0^t\Big(\frac{T-t}{T-s}\Big)^{A_p}\Sigma_p\,\dd B_s \right)\\
             & = \int_0^t\Big(\frac{T-t}{T-s}\Big)^{A_j}\Sigma_j\,\dd B_s,
             \qquad t\in[0,T),
 \end{align*}
 which yields \eqref{help4}.
Note that the last but one equality follows by the construction of a multidimensional It\^{o} integral.
Namely, by a multidimensional version of Theorem 4.7.1 in Kuo \cite{Kuo}, we have
 \begin{align}\label{Kuo}
  \int_0^t\Big(\frac{T-t}{T-s}\Big)^A\Sigma\,\dd B_s
    = \lim_{n\to\infty}
       \sum_{k=1}^n \Big(\frac{T-t}{T-s_{k-1}}\Big)^A\Sigma (B_{s_k} - B_{s_{k-1}})
    \qquad \text{in \ $L^2$,}
 \end{align}
 where $\{0=s_0 <s_1<\cdots < s_n=t\}$ is a partition of $[0,t]$ with
  $\max_{1\leq k\leq n}(s_k-s_{k-1})\to0$ as $n\to\infty$.
Using \eqref{help6} we have
\begin{align*}
  \Big(\frac{T-t}{T-s_{k-1}}\Big)^A\Sigma (B_{s_k} - B_{s_{k-1}})
    =  \left( \Big(\frac{T-t}{T-s_{k-1}}\Big)^{A_j}\Sigma_j(B_{s_k} - B_{s_{k-1}})\right)_{j=1,\ldots,p}
 \end{align*}
 for $k=1,\ldots,n$, and hence, again by a multidimensional version of Theorem 4.7.1 in Kuo \cite{Kuo} we have,
 \begin{align*}
    &\lim_{n\to\infty}
       \sum_{k=1}^n \Big(\frac{T-t}{T-s_{k-1}}\Big)^A\Sigma (B_{s_k} - B_{s_{k-1}})\\
     & \quad = \lim_{n\to\infty} \sum_{k=1}^n
       \left( \Big(\frac{T-t}{T-s_{k-1}}\Big)^{A_j}\Sigma_j(B_{s_k} - B_{s_{k-1}})\right)_{j=1,\ldots,p}\\
    & \quad = \left( \int_0^t\Big(\frac{T-t}{T-s}\Big)^{A_1}\Sigma_1\,\dd B_s ,
             \ldots, \int_0^t\Big(\frac{T-t}{T-s}\Big)^{A_p}\Sigma_p\,\dd B_s \right)
                 \qquad \text{in \ $L^2$.}
 \end{align*}
Since the limit of a $L^2$-convergent sequence is almost surely well-defined, together with \eqref{Kuo} we get
 \[
    \int_0^t\Big(\frac{T-t}{T-s}\Big)^A\Sigma\,\dd B_s
      = \left( \int_0^t\Big(\frac{T-t}{T-s}\Big)^{A_1}\Sigma_1\,\dd B_s ,
             \ldots, \int_0^t\Big(\frac{T-t}{T-s}\Big)^{A_p}\Sigma_p\,\dd B_s \right)
              \qquad \text{a.s.,}
 \]
 implying the statement.
\end{proof}

Note that, by Lemma \ref{lemma_spectral1}, $(X_t^{[j]})_{t\in[0,T)}$ structurally has the same integral representation \eqref{mintrep}
 but with real spectrally simple exponent $A_j$ whose eigenvalues all have the same real part $a_j$.
Concluding, we only need to consider real spectrally simple exponents $A$  to decide whether $X_t\to0$ almost surely as $t\uparrow T$ or not.

We will need the following result on the growth behavior of the exponential operator $t^{A_j}$ near the origin $t=0$ for $j=1,\ldots,p$.
For a matrix $Q\in\RR^{d_j\times d_j}$, now we choose the associated matrix norm
 \[
   \Vert Q\Vert := \sup\big\{ \Vert Qy\Vert : \Vert y\Vert = 1, \, y\in\RR^{d_j} \big\}
 \]
 with respect to the standard Euclidean norm $\Vert y\Vert$ for $y\in\RR^{d_j}$.

\begin{lemma}\label{specbound}
For every $j=1,\ldots,p$ and every $\varepsilon>0$, there exists a constant $K\in(0,\infty)$ such that for all $0<t\leq T$ we have
 $$
 \|t^{A_{j}}\|\leq K\,t^{a_{j}-\varepsilon}\quad\text{and }\quad\|t^{-A_{j}}\|\leq K\,t^{-(a_{j}+\varepsilon)}.
 $$
\end{lemma}

\begin{proof}
By Corollary 2.2.5 in Meerschaert and Scheffler \cite{MS},
 if $\beta<a_j$, then $t^{-\beta}\Vert t^{A_j}x\Vert\to \infty$ as $t\to\infty$ uniformly in $x\in\RR^{d_j}$
 from compact subsets of $\pi_j(\RR^d)\setminus \{0\}$.
Hence for all $\varepsilon>0$,
 \[
   t^{-(a_j-\varepsilon)}\Vert t^{A_j}x\Vert\to 0 \qquad \text{as \ $t\to\infty$}
 \]
 uniformly in $x\in\RR^{d_j}$ from compact subsets of $\pi_j(\RR^d)\setminus\{0\}$, i.e.,
 \[
   \lim_{t\to\infty}\sup_{x\in R_j} t^{-(a_j-\varepsilon)} \Vert t^{A_j}x\Vert=0
 \]
 for all compact subsets $R_j$ of $\pi_j(\RR^d)\setminus\{0\}$.
Then, by choosing $R_j:= \{ x\in \pi_j(\RR^d) : \Vert x\Vert =1\}$, and using the definition of the norm,
 we have
 \[
   \lim_{t\to\infty} \sup_{x\in R_j} t^{-(a_j-\varepsilon)} \Vert t^{A_j}x\Vert
      =\lim_{t\to\infty} t^{-(a_j-\varepsilon)} \Vert t^{A_j}\Vert
      = 0.
 \]
Since a convergent sequence is bounded, we have
 \[
   \sup_{t\in(0,T]} t^{-(a_j-\varepsilon)} \Vert t^{A_j}\Vert =:K_j' <\infty,
   \quad j=1,\ldots,p.
 \]
Hence with $K':=\max\{K_1',\ldots,K_p'\}$,  we have $\Vert t^{A_j}\Vert \leq K' t^{a_j-\varepsilon}$ for all $t\in(0,T]$ and $j=1,\ldots,p$.

Similarly, by Corollary 2.2.5 in Meerschaert and Scheffler \cite{MS},
 if $\beta>a_j$, then $t^{-\beta}\Vert t^{A_j}x\Vert\to 0$ as $t\to\infty$ uniformly in $x\in\RR^{d_j}$ from compact subsets
 of $\pi_j(\RR^d)$.
Hence, with the same arguments as above, there exists $K''>0$ such that $\Vert t^{-A_j}\Vert \leq K'' t^{-(a_j+\varepsilon)}$ for all $t\in(0,T]$
 and $j=1,\ldots,p$.

Finally, one can choose $K:=\max\{K',K''\}$.
\end{proof}

We note that in Lemma \ref{specbound} one can use any matrix norm on $\RR^{d_j\times d_j}$ (since any two matrix norms on $\RR^{d_j\times d_j}$
 are equivalent).

\subsection{SLLN and LIL for martingales on $[0,T)$.}

Recall the integral representation \eqref{mintrep} of the solution $(X_t)_{t\in[0,T)}$ of \eqref{msde} with $X_0=0$. We may write
\begin{equation}\label{martdec}
X_t=(T-t)^AM_t\quad\text{with}\quad M_t:=\int_0^t(T-s)^{-A}\Sigma\,\dd B_s,\quad t\in[0,T).
\end{equation}
Here $(M_t)_{t\in[0,T)}$ is a continuous square-integrable martingale whose $i$-th coordinate $(M_t^{(i)})_{t\in[0,T)}$ has quadratic variation process given by
 \begin{equation}\label{qvar}
 \langle M^{(i)}\rangle_t=\int_0^t\big\| e_i^\top(T-s)^{-A}\Sigma \big\|^2\,\dd s,\quad t\in[0,T),
 \end{equation}
 for every $i=1,\ldots,d$, where $\{e_1,\ldots,e_d\}$ denotes the canonical basis of $\RR^d$.
Note that $(\langle M^{(i)}\rangle_t)_{t\in[0,T)}$ is a continuous deterministic function.

We call the attention that from now on the superscripts in curved brackets denote coordinates rather than spectral components denoted by superscripts with squared brackets as in Section \ref{sebsection_spectral}.

Usually, the strong law of large numbers for martingales is formulated as a limit theorem as $t\to\infty$.
In our case we need to consider the limiting behavior as $t\uparrow T$.
Due to the strictly increasing and continuous time change $t(s)=(2T/\pi)\arctan s$, $s\geq 0$ (which is a bijection between
 $[0,\infty)$ and $[0,T)$),
 \ we get that $(\widetilde M_s:=M_{t(s)})_{s\geq0}$ is a continuous square-integrable martingale with respect to the filtration
  $(\widetilde{\mathcal F}_s:=\mathcal F_{t(s)})_{s\geq0}$ and we can easily adopt the following well-known versions of the strong law of large numbers for continuous square-integrable martingales.
\begin{lemma}\label{lln1}
 If $\lim_{t\uparrow T}\langle M^{(i)}\rangle_t<\infty$ for every $i=1,\ldots,d$, then
 \[
   P\left(\lim_{t\uparrow T}M_t\;\text{exists} \right)=1.
 \]
\end{lemma}
 For the proof we refer to Proposition 4.1.26 together with Proposition 5.1.8 in \cite{RevYor}.
 \begin{lemma}\label{lln2}
 Let $f:[x_0,\infty)\to(0,\infty)$ be an increasing function, where $x_0>0$ such that $\int_{x_0}^\infty f(x)^{-2}\,\dd x<\infty$.
 If $\lim_{t\uparrow T}\langle M^{(i)}\rangle_t=\infty$ for some $i\in\{1,\ldots,d\}$, then
 \[
  P\left(\lim_{t\uparrow T}\frac{M^{(i)}_t}{f(\langle M^{(i)}\rangle_t)}=0\right)=1.
 \]
 \end{lemma}
For the proof we refer to Exercise 5.1.16 in \cite{RevYor} or to Theorem 2.3 in \cite{BarPap}.

Next we present a law of the iterated logarithm for $(M_t)_{t\in[0,T)}$.

\begin{lemma}\label{MLIL}
If $P(\lim_{t \uparrow T} \langle M^{(i)} \rangle_t = \infty)=1$ for some $i\in\{1,\ldots,d\}$, then
 \begin{align*}
  & P\left( \limsup_{t \uparrow T}
              \frac{M^{(i)}_t}{\sqrt{2 \langle M^{(i)} \rangle_t \ln(\ln \langle M^{(i)} \rangle_t)}}
        = 1 \right)\\
  &\quad =
    P\left( \liminf_{t \uparrow T}
               \frac{M^{(i)}_t}{\sqrt{2 \langle M^{(i)} \rangle_t \ln(\ln \langle M^{(i)} \rangle_t)}}
        = - 1 \right)
   = 1 .
 \end{align*}
\end{lemma}

Lemma \ref{MLIL} follows by Exercise 1.15 in Chapter V of Revuz and Yor \cite{RevYor}.

\section{Bridge property}\label{Section_bridge_property}

Let $\RSpec(A):=\{\Re\lambda:\,\lambda\in\Spec(A)\}$ be the collection of distinct real parts of the eigenvalues of the matrix $A$,
 where $\Spec(A)$ denotes the set of eigenvalues of $A$.
If there exists $\lambda\in\Spec(A)$ with $\Re\lambda\leq0$, then the process $(X_t)_{t\in[0,T)}$ defined by \eqref{mintrep}
 with initial condition \ $X_0=0\in\RR^d$ does not fulfill that $X_t$ converges to some deterministic $d$-dimensional vector
 almost surely as $t\uparrow T$ in general.
This fact is known for the one-dimensional situation $d=1$ from Remark 3.5 in \cite{BarPap}.
To give an explicit multidimensional example, we consider a $d\times d$ matrix $A$ having only purely imaginary eigenvalues.
\begin{example}\label{Example_no_bridge}
Let $\Sigma=I_d$ be the $d\times d$ identity matrix and let $A\in\RR^{d\times d}$ be a skew symmetric matrix, i.e.\ $A^\top = -A$. Then all the non-zero eigenvalues of $A$ are purely imaginary and $r^A$ is an orthogonal matrix for every $r>0$.
Due to the invariance of the incremental distributions of a standard Wiener process with respect to orthogonal transformations, using \eqref{Kuo} one can easily derive that the distributions of $X_t$ and $B_t$ coincide for every $t\in[0,T)$.
Hence $X_t$ converges in distribution to $B_T$ as $t\uparrow T$, which shows that it cannot hold that $X_t$ converges almost surely to some deterministic $d$-dimensional vector as $t\uparrow T$.
\end{example}

Our next result is about the limit behavior of the quadratic variation processes $\langle M^{(i)}\rangle_t$ as $t\uparrow T$
 for $i=1,\ldots,d$.

\begin{lemma}\label{lemma_quadratic_variation}
If $\RSpec(A)\subseteq(0,1/2)$, then for all $i=1,\ldots,d$, the quadratic variation process
 $(\langle M^{(i)}\rangle_t)_{t\in[0,T)}$ is bounded.
If $\RSpec(A)\subseteq(1/2,\infty)$ and $\Sigma$ has full rank $d$ (and consequently $m\geq d$),
 then $\lim_{t\uparrow T}\langle M^{(i)}\rangle_t = \infty$ for all $i=1,\ldots,d$.
\end{lemma}

\begin{proof}
As explained in Section 2.1, we only need to consider real spectrally simple matrices $A$ with $\RSpec(A)=\{a\}$ for some $a>0$.
Note that if $\Sigma$ has full rank $d$, then $\Sigma_j$ has full rank $d_j$ for all $j=1,\ldots,p$
 (due to $\pi_j(\Sigma y) = \Sigma_j y$, $y\in\RR^m$).
We distinguish between the following two cases.

{\it Case 1:} $a\in(0,\frac12)$. Let $\beta\in(a,1/2)$.
Then according to \eqref{qvar} and Lemma \ref{specbound}
 with $a_j:=a$ and $\varepsilon:=\beta - a_j>0$, there exists a constant $K>0$ such that
\begin{align}\label{help5}
  \begin{split}
 \langle M^{(i)}\rangle_t & =\int_0^t\big\| e_i^\top(T-s)^{-A}\Sigma \big\|^2\,\dd s\leq\int_0^t\|e_i\|^2\|(T-s)^{-A}\|^2\|\Sigma \|^2\,\dd s\\
  & \leq K^2\Vert\Sigma\Vert^2 \int_0^t(T-s)^{-2\beta}\,\dd s
    = K^2 \Vert\Sigma\Vert^2 \,\frac{T^{1-2\beta}-(T-t)^{1-2\beta}}{1-2\beta}
 \end{split}
 \end{align}
for $t\in[0,T)$ and $i=1,\ldots,d$, which shows that the quadratic variation process $(\langle M^{(i)}\rangle_t)_{t\in[0,T)}$ is bounded due to $\beta<1/2$.

{\it Case 2:} $a>1/2$ and $\Sigma$ has full rank $d$.
By \eqref{qvar}, we get
 \begin{align*}
  \langle M^{(i)}\rangle_t
      = \int_0^t \Vert e_i^\top(T-s)^{-A}\Sigma\Vert^2\,\dd s
     = \int_0^t \Vert \Sigma^\top (T-s)^{-A^\top}e_i\Vert^2\,\dd s
 \end{align*}
for $i=1,\ldots,d$ and $t\in[0,T)$.
Since $\Sigma$ has full rank $d$, we have $\Sigma\Sigma^\top$ is invertible
 and
 \begin{align*}
  \Vert(T-s)^{-A^\top}e_i\Vert^2
       \leq \Vert (\Sigma\Sigma^\top)^{-1} \Sigma\Vert^2
            \Vert \Sigma^\top (T-s)^{-A^\top}e_i\Vert^2
            =: C \Vert \Sigma^\top (T-s)^{-A^\top}e_i\Vert^2
 \end{align*}
 for all $s\in[0,T)$ and $i=1,\ldots,d$, where $C>0$.
This yields that
 \[
   \langle M^{(i)}\rangle_t \geq C^{-1} \int_0^t \Vert(T-s)^{-A^\top}e_i\Vert^2\,\dd s,
   \qquad t\in[0,T),\;\; i=1,\ldots,d.
 \]
By Theorem 2.2.4 in Meerschaert and Scheffler \cite{MS}, for all $\varepsilon>0$ and $i\in\{1,\ldots,d\}$,
 one can choose a $t_0\in(0,T)$ such that
 \[
   (T-s)^{a-\varepsilon} \Vert (T-s)^{-A^\top}e_i\Vert\geq 1,
   \qquad \forall \,s\in[t_0,T).
 \]
Hence for all $t\geq t_0$, $t\in[0,T)$ and $i\in\{1,\ldots,d\}$,
 \begin{align*}
  \langle M^{(i)}\rangle_t
     \geq C^{-1} \int_0^{t_0} \Vert(T-s)^{-A^\top}e_i\Vert^2\,\dd s
          + C^{-1} \int_{t_0}^t (T-s)^{-2a + 2\varepsilon}\,\dd s
       \to\infty \quad \text{as \ $t\uparrow T$,}
 \end{align*}
 provided that $a-\varepsilon>1/2$.
Since $a>1/2$, one can choose such an $\varepsilon$, which yields that $\lim_{t\uparrow T}\langle M^{(i)}\rangle_t =\infty$ for every $i\in\{1,\ldots,d\}$.
\end{proof}

\begin{remark}\label{rem2}
We conjecture that $\lim_{t\uparrow T}\langle M^{(i)}\rangle_t = \infty$
 for every $i\in\{1,\ldots,d\}$ in case $\RSpec(A) = \{1/2\}$.
However, we cannot address a precise argument.
Note that in dimension 1 this holds; see the proof of Lemma 3.1 in Barczy and Pap \cite{BarPap}.
Fortunately, for proving the bridge property of $(X_t)_{t\in[0,T)}$ we do not need any information about the limit behavior of the quadratic variation process in case $A$ has eigenvalues with real part all equal to $\frac12$, see the proof of Theorem \ref{Theorem1} below.
\end{remark}

Now we are ready to formulate our main result.

\begin{theorem}\label{Theorem1}
Let us suppose that $\Sigma$ has full rank $d$ (and consequently $m\geq d$).
If $\RSpec(A)\subseteq(0,\infty)$, then the process
\begin{equation}\label{mbridge}
\widehat X_t
   :=\begin{cases}
          \displaystyle\int_0^t\Big(\frac{T-t}{T-s}\Big)^A\Sigma\,\dd B_s & \text { if }t\in[0,T),\\
          0 & \text { if }t=T
    \end{cases}
\end{equation}
is a centered Gauss process with almost surely continuous sample paths.
\end{theorem}
\begin{remark}\label{rem1}
Note that the condition $\RSpec(A)\subseteq(0,\infty)$ is equivalent to $t^A\to0\in\RR^{d\times d}$ as $t\downarrow0$.
 We call the attention that the condition that $\Sigma$ has full rank $d$ in Theorem \ref{Theorem1}
 is needed only for the case $\RSpec(A)\cap [1/2,\infty) \neq \emptyset$; see the proof given below. Moreover, as mentioned in the Introduction, the assumption that $\Sigma$ has full rank $d$ is only a minor restriction to the generality of Theorem \ref{Theorem1}.%
\end{remark}

\begin{proof}[Proof of Theorem \ref{Theorem1}]
The process $(\widehat X_t)_{t\in[0,T]}$ given by \eqref{mbridge} is a centered Gauss process even without the restriction on $A$, see, e.g., Karatzas and Shreve \cite[Problem 5.6.2]{KS}.
As explained in Section 2.1, we only need to consider real spectrally simple matrices $A$ with $\RSpec(A)=\{a\}$ for some $a>0$.
We distinguish between the following three cases.

{\it Case 1:} $a\in(0,\frac12)$. By Lemma \ref{lemma_quadratic_variation}, for all \ $i\in\{1,\ldots,d\}$
 the quadratic variation process $(\langle M^{(i)}\rangle_t)_{t\in[0,T)}$ is bounded.
Hence, since it is non-decreasing, we have
 $\lim_{t\uparrow T}\langle M^{(i)}\rangle_t$ exists and is finite for every $i=1,\ldots,d$.
 An application of Lemma \ref{lln1} shows that $\lim_{t\uparrow T}M_t$ exists almost surely.
Since $a>0$, by Lemma \ref{specbound}, $(T-t)^A\to0\in\RR^{d\times d}$ as $t\uparrow T$ and by \eqref{martdec}
 we get $\widehat X_t=(T-t)^AM_t\to0 = \widehat X_T$ almost surely as $t\uparrow T$.

{\it Case 2:} $a>\frac12$.
Let $\varepsilon>0$ and $\delta:=\frac{2a-1+\varepsilon}{2(2a-1)+\varepsilon}$.
We define $f:[1,\infty)\to\RR$ by $f(x):=x^{\delta}$ for $x\geq1$.
Then $\delta>\frac12$, and $\int_1^\infty f(x)^{-2}\,dx = (2\delta -1)^{-1}<\infty$.
By Lemma \ref{lemma_quadratic_variation}, we have
 $\langle M^{(i)}\rangle_t\to\infty$ as $t\uparrow T$, and using Lemma \ref{lln2} we get
 $M^{(i)}_t/f(\langle M^{(i)}\rangle_t)\to 0$ as $t\uparrow T$ almost surely
 for every $i=1,\ldots,d$.
By \eqref{martdec}, we have for sufficiently large $t$
 \begin{align*}
 X_t^{(i)} & =\sum_{j=1}^d\big((T-t)^A\big)_{ij}f(\langle M^{(j)}\rangle_t)\,\frac{M_t^{(j)}}{f(\langle M^{(j)}\rangle_t)}
 \end{align*}
 for every $i=1,\ldots,d$.
To prove that $\widehat X_t\to 0 = \widehat X_T$ as $t\uparrow T$ almost surely,
 it is enough to show that $\big((T-t)^A\big)_{ij}f(\langle M^{(j)}\rangle_t)$,
 $t\in[0,T)$, is bounded for every $i,j\in\{1,\ldots,d\}$.
Since any two norms on $\RR^{d\times d}$ are equivalent,
 one can choose a constant $C>0$ such that for every matrix $D\in\RR^{d\times d}$,
 \[
   \Vert D\Vert_1:=\max_{j=1,\ldots,d} \sum_{i=1}^d \vert d_{i,j}\vert
    \leq C \Vert D\Vert.
 \]
Similarly to \eqref{help5}, by Lemma \ref{specbound} with $a_j=a>\frac12$, we get for some constant $K>0$,
 \begin{align}\label{help12}
 \langle M^{(j)}\rangle_t\leq K^2
   m\Vert\Sigma\Vert^2\,\frac{T^{1-2(a+\varepsilon)}-(T-t)^{1-2(a+\varepsilon)}}{1-2(a+\varepsilon)}
    \leq K^2 m\Vert\Sigma\Vert^2\,\frac{(T-t)^{1-2(a+\varepsilon)}}{2(a+\varepsilon)-1}
 \end{align}
 for $t\in[0,T)$ and $j=1,\ldots,d$.
Then
 \begin{align*}
  \big|\big((T-t)^A\big)_{ij}f(\langle M^{(j)}\rangle_t)\big|
  & \leq \sum_{i=1}^d\big|\big((T-t)^A\big)_{ij}\big|\,\big|f(\langle M^{(j)}\rangle_t)\big|\\
  & \leq \|(T-t)^A\|_1\,f\Big(K^2 m\Vert\Sigma\Vert^2\,\frac{(T-t)^{1-2(a+\varepsilon)}}{2(a+\varepsilon)-1}\Big)\\
  &\leq C \Vert (T-t)^A\Vert \left(\frac{ K^2m\Vert\Sigma\Vert^2}{2(a+\varepsilon) -1}\right)^\delta
          (T-t)^{\delta(1-2(a+\varepsilon))} \\
  & \leq \widetilde C (T-t)^{\delta(1-2(a+\varepsilon)) + a-\varepsilon},
  \quad t\in[0,T)
 \end{align*}
 for some constant $\widetilde C>0$ (depending on $\varepsilon$) and $i,j\in\{1,\ldots,d\}$,
 where the last inequality follows again by Lemma \ref{specbound} with $a_j:=a$.
Our choice of $\delta>0$ yields that
 \begin{align*}
  \delta(1-2(a+\varepsilon)) + a-\varepsilon
    &= a-\varepsilon - \frac{(2(a+\varepsilon) -1)(2a-1+\varepsilon)}
                {2(2a + \varepsilon/2 -1)}\\
    &\to a - \frac{(2a-1)^2}{2(2a-1)}
        = \frac{1}{2}>0 \qquad \text{as $\varepsilon\downarrow 0$,}
 \end{align*}
 which shows the desired boundedness of $\big((T-t)^A\big)_{ij}f(\langle M^{(j)}\rangle_t)$,
 $t\in[0,T)$, for sufficiently small \ $\varepsilon>0$.

{\it Case 3:} $a=\frac12$.
Decompose $\{1,\ldots,d\}=\cI_1\cup \cI_2$ in such a way that the (deterministic) quadratic variation process
 $(\langle M^{(i)}\rangle_t)_{t\in[0,T)}$ is bounded for every $i\in \cI_1$ and $\langle M^{(i)}\rangle_t\to\infty$
 as $t\uparrow T$ for every $i\in \cI_2$.  
For $i\in\{1,\ldots,d\}$ we obtain for sufficiently large $t$
\begin{align*}
X_t^{(i)} & =\sum_{j=1}^d\big((T-t)^A\big)_{ij}M_t^{(j)}\\
& =\sum_{j\in \cI_1}\big((T-t)^A\big)_{ij}M_t^{(j)}+\sum_{j\in \cI_2}\big((T-t)^A\big)_{ij}
       f(\langle M^{(j)}\rangle_t)\,\frac{M_t^{(j)}}{f(\langle M^{(j)}\rangle_t)}.
\end{align*}
Since $a>0$, by Lemma \ref{specbound}, $(T-t)^A\to 0\in\RR^{d\times d}$ as $t\to T$, and hence
 as in Case 1, we get that the first sum on the right-hand side tends to 0 as $t\uparrow T$ almost surely.
The second sum on the right-hand side tends to $0$ as well, since $\delta=(2a-1+\varepsilon)/(2(2a-1)+\varepsilon) =1 > 1/2$,
 $1-2(a+\varepsilon) = -2\varepsilon<0$ and $\delta(1-2(a+\varepsilon)) + a - \varepsilon = 1/2- 3\varepsilon>0$
 for $\varepsilon\in(0,1/6)$, and one can apply the method of Case 2 described above.
\end{proof}

\section{Asymptotic behavior of the bridge}\label{Section_asymptotic}

In this section we study asymptotic behavior of the sample paths of the operator scaled Wiener bridge
 $(X_t)_{t\in[0,T)}$ given by \eqref{msde} with initial condition $X_0=0$.

Our first result is a partial generalization of Theorem 3.4 in Barczy and Pap \cite{BarPap}.

\begin{prop}\label{prop_sample_path}
If $\RSpec(A)\subseteq(0,1/2)$, then
 \begin{align}\label{help2}
    P\Big(\lim_{t\uparrow T} (T-t)^{-A}X_t = M_T \Big)=1,
 \end{align}
 where $M_T$ is a $d$-dimensional normally distributed random variable.
Consequently, for all $\widetilde A\in\RR^{d\times d}$ with $A\widetilde A = \widetilde A A$, we have
 \begin{align}\label{help3a}
    P\Big(\lim_{t\uparrow T} (T-t)^{-\widetilde A}X_t = 0 \Big)=1 \; & \text{ if \ $\RSpec(A - \widetilde A)\subseteq(0,\infty)$,}\\
     P\Big(\lim_{t\uparrow T} \Vert (T-t)^{-\widetilde A}X_t \Vert = \infty\Big) =1
        \; & \text{if \ $\RSpec(A - \widetilde A)\subseteq(-\infty,0)$.}\label{help3b}
 \end{align}
\end{prop}

\begin{proof}
By \eqref{martdec},
 \[
   (T-t)^{-A} X_t = M_t = \int_0^t (T-s)^{-A}\Sigma\,\dd B_s,\quad t\in[0,T),
 \]
 and since $\RSpec(A)\subseteq(0,1/2)$, by the proof of Theorem \ref{Theorem1} (Case 1), we have $M_T:=\lim_{t\uparrow T} M_t$
 exists almost surely.
Since $M_t$ is a $d$-dimensional normally distributed random variable for all $t\in[0,T)$,
 and normally distributed random variables can converge in distribution only to a normally distributed random variable
 (following directly, e.g., from Theorem 3.1.16 in Meerschaert and Scheffler \cite{MS}),
 we have that $M_T$ is normally distributed, which yields \eqref{help2}.
Hence for all $A,\widetilde A\in\RR^{d\times d}$ with $A\widetilde A = \widetilde A A$
 we have
 \begin{align}\label{help7}
    (T-t)^{-\widetilde A} X_t = (T-t)^{A-\widetilde A} M_t,\quad t\in[0,T).
 \end{align}

If $\RSpec(A - \widetilde A)\subseteq(0,\infty)$, then, by Remark \ref{rem1},
 \[
  \lim_{t\uparrow T}(T-t)^{A-\widetilde A} = 0\in\RR^{d\times d},
 \]
 and together with $P(\lim_{t\uparrow T} M_t = M_T)=1$, we have \eqref{help3a}.

If $\RSpec(A - \widetilde A)\subseteq(-\infty,0)$, then 
by \eqref{help7}, $\Vert (T-t)^{-\widetilde A} X_t\Vert = \Vert (T-t)^{A-\widetilde A} M_t\Vert$ for $t\in[0,T)$.
Since $P(M_T=0)=0$, for almost every $\omega\in\Omega$ one can choose  $\eta(\omega)>0$ such that
 $0\notin \{x\in\RR^d : \Vert M_T(\omega) - x\Vert <\eta(\omega)\}$.
Further, since \ $\lim_{t\uparrow T} M_t = M_T$ \ almost surely one can choose a $t_0\in[0,T)$ such that
 $M_t(\omega)\in\{x\in\RR^d : \Vert M_T(\omega) - x\Vert <\eta(\omega)\}$ for $t\in[t_0,T)$.
Using Theorem 2.2.4 in \cite{MS} with $\alpha=0$ and with the compact set
 $S(\omega):=\{x\in\RR^d : \Vert M_T(\omega) - x\Vert \leq \eta(\omega)\}$,
 we have $\Vert s^{-(A-\widetilde A)} x\Vert\to \infty$ as $s\to\infty$ for all $x\in S(\omega)$,
 and hence $\Vert (T-t)^{A-\widetilde A} M_t(\omega)\Vert\to \infty$ as $t\uparrow T$.
This yields \eqref{help3b}.
\end{proof}

Recall the spectral decomposition of the process $(X_t)_{t\in[0,T)}$, see Lemma \ref{lemma_spectral1}.
For the spectral components, one can get the following precise asymptotic result.

\begin{theorem}
If $\RSpec(A)\subseteq(0,\infty)$ and $\Sigma$ has full rank $d$ (and consequently $m\geq d$),
 then for all $\varepsilon>0$,
 \begin{align}\label{help8}
    P\Big(\lim_{t\uparrow T} (T-t)^{-\min(a_j,1/2) + \varepsilon} \Vert X_t^{[j]}\Vert = 0 \Big)=1,\\\label{help9}
    P\Big(\limsup_{t\uparrow T} (T-t)^{-\min(a_j,1/2) - \varepsilon} \Vert X_t^{[j]}\Vert = \infty \Big)=1,
 \end{align}
 where \ $a_1<\cdots<a_p$ \ denote the distinct real parts of the eigenvalues of $A$ and $(X_t^{[j]})_{t\in[0,T)}$,
 $j=1,\ldots,p$,  are the corresponding spectral components of $(X_t)_{t\in[0,T)}$, see Section \ref{sebsection_spectral}.
Further, if $\RSpec(A)\subseteq(0,1/2)$, then \eqref{help9} can be strengthened to
 \begin{align}\label{help10}
  P\Big(\lim_{t\uparrow T} (T-t)^{-a_j - \varepsilon} \Vert X_t^{[j]}\Vert = \infty \Big)=1.
 \end{align}
\end{theorem}

\begin{proof}%
As explained in Section 2.1, we only need to consider real spectrally simple matrices $A$ with $\RSpec(A)=\{a\}$ for some $a>0$.
Then $p=1$ and $X^{[1]}=X$.
Note that if $\Sigma$ has full rank $d$, then $\Sigma_j$ has full rank $d_j$ for all $j=1,\ldots,p$.

{\it Case 1:} $a\in(0,\frac12)$. To prove \eqref{help8} we use \eqref{help3a} with $\widetilde A:=(a-\varepsilon)I_d$.
Then $A\widetilde A = \widetilde A A$ and $\RSpec(A - \widetilde A) = \{\varepsilon \}\subseteq(0,\infty)$.
 Hence we have $P(\lim_{t\uparrow T} (T-t)^{-a+\varepsilon}X_t =0 )=1$ by \eqref{help3a}.
To prove \eqref{help10} we apply \eqref{help3b} with $\widetilde A:=(a+\varepsilon)I_d$.
Then $A\widetilde A = \widetilde A A$ and $\RSpec(A - \widetilde A) = \{-\varepsilon \}\subseteq(-\infty,0)$. Hence $P(\lim_{t\uparrow T} (T-t)^{-(a+\varepsilon)} \Vert X_t\Vert = \infty )=1$ by \eqref{help3b}.

{\it Case 2:} $a>\frac12$.
We use the well-known equality
 \[
   \Vert x\Vert = \sup_{\theta\in\RR^d, \Vert\theta\Vert=1} \vert\langle x,\theta\rangle\vert,
   \qquad x\in\RR^d,
  \]
 where $\langle \cdot,\cdot\rangle$ denotes the usual Euclidean inner product on $\RR^d$.
Hence to prove \eqref{help8} it is sufficient to show
 \[
    P\left(\lim_{t\uparrow T} (T-t)^{-1/2 + \varepsilon}
       \sup_{\theta\in\RR^d, \Vert\theta\Vert=1} \vert\langle X_t,\theta\rangle\vert
        = 0\right)=1.
 \]
First, we verify that for all $\varepsilon>0$ there exists some $t_0\in[0,T)$ such that
 \begin{align}\label{help11}
  (T-t)^{a+\varepsilon}
    \leq \sup_{\theta\in\RR^d, \Vert\theta\Vert=1} \Vert \theta^\top (T-t)^A  \Vert
     \leq (T-t)^{a-\varepsilon},
    \qquad t\in[t_0,T).
 \end{align}
Indeed, using Theorem 2.2.4 in \cite{MS} with $\alpha:=-(a+\varepsilon)$ and $\beta:=-(a-\varepsilon)$, respectively, we get
 \begin{align*}
  &\lim_{t\uparrow T} \Big(((T-t)^{-1})^{-\alpha} \sup_{\theta\in\RR^d, \Vert\theta\Vert=1}
                      \Vert ((T-t)^{-1})^{-A^\top} \theta \Vert \Big) =\infty,\\
  &\lim_{t\uparrow T} \Big(((T-t)^{-1})^{-\beta} \sup_{\theta\in\RR^d, \Vert\theta\Vert=1}
                      \Vert ((T-t)^{-1})^{-A^\top} \theta \Vert \Big) =0.
 \end{align*}
This yields \eqref{help11} taking into account that $\Vert ((T-t)^{-1})^{-A^\top} \theta\Vert = \Vert \theta^\top (T-t)^A\Vert$,
 $t\in[0,T)$.
Note also that for all $t\in[t_0,T)$ and $\theta\in\RR^d$ with $\Vert\theta\Vert=1$,
 \begin{align}\label{help15}
   \langle X_t,\theta\rangle
     = \theta^\top X_t
     = \theta^\top (T-t)^A M_t
     = \Vert \theta^\top (T-t)^A \Vert \frac{\theta^\top (T-t)^A}{\Vert \theta^\top (T-t)^A \Vert} M_t.
 \end{align}
Due to $a>1/2$ and $\Sigma$ has full rank $d$, by Lemma \ref{lemma_quadratic_variation} we have
 $\lim_{t\uparrow T} \langle M^{(i)}\rangle_t = \infty$ for $i=1,\ldots,d$, and hence, by \eqref{help15},
 for sufficiently large $t$,
 \begin{align*}
  &(T-t)^{-\min(a,1/2) + \varepsilon} \sup_{\theta\in\RR^d, \Vert\theta\Vert=1} \langle X_t,\theta\rangle\\
  &\quad \leq (T-t)^{-1/2 + \varepsilon} \sum_{i=1}^d
      \sup_{\theta\in\RR^d, \Vert\theta\Vert=1} \left( \Vert \theta^\top (T-t)^A\Vert
      \left\vert \left(\frac{\theta^\top (T-t)^A}{\Vert\theta^\top (T-t)^A\Vert}\right)_i \right\vert\right)\\
  &\phantom{\quad\leq(T-t)^{-1/2 + \varepsilon} \sum_{i=1}^d }\times
      f(\langle M^{(i)}\rangle_t)
      \frac{\vert M^{(i)}_t \vert}{f(\langle M^{(i)}\rangle_t)},
 \end{align*}
 where $f(x):=x^\delta$, $x\geq 1$, with $\delta:=\frac{2a-1+\varepsilon'}{2(2a-1) + \varepsilon'}>1/2$ for some $\varepsilon'>0$.
Note that, by the proof of Theorem \ref{Theorem1} (Case 2),
 $\int_1^\infty (f(x))^{-2}\,\dd x<\infty$.
Using the strong law of large numbers for continuous local martingales (see, e.g., Lemma \ref{lln1}) we have
 \[
    P\left(\lim_{t\uparrow T} \frac{ M^{(i)}_t }{f(\langle M^{(i)}\rangle_t)} =0 \right)=1, \qquad i=1,\ldots,d.
 \]
Hence in order to show \eqref{help8}, it is sufficient to check that the function
 \begin{align}\label{help13}
  (T-t)^{-1/2 + \varepsilon} f(\langle M^{(i)}\rangle_t)
      \sup_{\theta\in\RR^d, \Vert\theta\Vert=1} \Vert \theta^\top (T-t)^A\Vert,
      \qquad t\in[0,T),
 \end{align}
 is bounded for $i=1,\ldots,d$.
By \eqref{help12} and \eqref{help11}, for all $\varepsilon''>0$ and $\varepsilon'''>0$, there exist some $t_0\in[0,T)$ and
 a constant $\widetilde C>0$ (depending on $\varepsilon'$ and $\varepsilon'''$) such that
 \begin{align*}
  (T-t)^{-1/2 + \varepsilon} f(\langle M^{(i)}\rangle_t)
   &\sup_{\theta\in\RR^d, \Vert\theta\Vert=1} \Vert \theta^\top (T-t)^A \Vert\\
   &\leq (T-t)^{-1/2 + \varepsilon} (T-t)^{a-\varepsilon''} \langle M^{(i)}\rangle_t^\delta\\
   &\leq \widetilde C (T-t)^{a-1/2 + \varepsilon - \varepsilon''} (T-t)^{\delta(1-2(a+\varepsilon'''))}\\
   &= \widetilde C (T-t)^{\widetilde\delta},
   \qquad t\in[t_0,T),
 \end{align*}
 where
 \[
   \widetilde\delta:= a-1/2 + \varepsilon - \varepsilon'' +  \frac{(2a-1+\varepsilon')(1-2(a+\varepsilon'''))}{2(2a-1) + \varepsilon'}.
 \]
Let us choose $\varepsilon':=\varepsilon/2$, $\varepsilon'':=\varepsilon'/2$ and $\varepsilon''':=\varepsilon'/4$, where $\varepsilon>0$.
Then
 \begin{align*}
  \widetilde\delta
  & = a - \frac{1}{2} + \varepsilon - \frac{\varepsilon'}{2}
       - \frac{(2a - 1 + \varepsilon')(2a - 1 + \varepsilon'/2)}{2(2a - 1 + \varepsilon'/2)}
    = a - \frac{1}{2} + \varepsilon - \frac{\varepsilon'}{2} - \left( a - \frac{1}{2} + \frac{\varepsilon'}{2}\right)\\
  & = \varepsilon - \frac{\varepsilon'}{2} - \frac{\varepsilon'}{2}
    = \varepsilon - \varepsilon'=\frac{\varepsilon}{2}>0,
 \end{align*}
 which yields that the function given in \eqref{help13} is bounded for all $i=1,\ldots,d$.

Note that the above argument for proving \eqref{help8} works also in the case $a=1/2$ for indices $i\in\{1,\ldots,d\}$
 with $\lim_{t\uparrow T} \langle M^{(i)}\rangle_t = \infty$.

Now we turn to prove \eqref{help9}.
Recall that for a vector $x\in\RR^d$ and a matrix $A\in\RR^{d\times d}$,
 \[
   \Vert x\Vert_1 = \sum_{i=1}^d\vert x_i\vert
   \quad \text{and}\quad
   \Vert A\Vert_1 = \max_{j=1,\ldots,d} \sum_{i=1}^d \vert a_{i,j}\vert.
 \]
Using that
 \[
   \Vert M_t\Vert_1 = \Vert (T-t)^{-A} (T-t)^A M_t\Vert_1
                   \leq \Vert (T-t)^{-A}\Vert_1 \Vert (T-t)^A M_t\Vert_1,
                   \quad t\in[0,T),
 \]
 we have for sufficiently large $t$,
 \begin{align*}
  &(T-t)^{-1/2-\varepsilon}  \Vert X_t\Vert_1
      = (T-t)^{-1/2-\varepsilon} \Vert (T-t)^A M_t\Vert_1
     \geq (T-t)^{-1/2-\varepsilon} \frac{\Vert M_t\Vert_1}{\Vert (T-t)^{-A}\Vert_1}\\
  &\quad = (T-t)^{-1/2-\varepsilon} \Vert (T-t)^{-A}\Vert_1^{-1} \sum_{i=1}^d \vert M^{(i)}_t\vert\\
  &\quad = \Vert (T-t)^{-A + (1/2+\varepsilon)I_d}\Vert_1^{-1}
       \sum_{i=1}^d \sqrt{2\langle M^{(i)}\rangle_t \ln\ln \langle M^{(i)}\rangle_t }
        \frac{\vert M^{(i)}_t \vert}{\sqrt{2\langle M^{(i)}\rangle_t \ln\ln \langle M^{(i)}\rangle_t}}.
 \end{align*}
If we show that for every $i=1,\ldots,d$,
 \[
    \lim_{t\uparrow T}\Vert (T-t)^{- A + (1/2+\varepsilon)I_d}\Vert_1^{-1} \sqrt{\langle M^{(i)}\rangle_t} =\infty,
 \]
 then \eqref{help9} follows, since $\Vert\cdot\Vert_1$ and $\Vert\cdot\Vert$ are equivalent norms on $\RR^d$,
 by Lemma \ref{MLIL} we have
 \[
    P\left(\limsup_{t\uparrow T} \frac{\vert M^{(i)}_t \vert}{\sqrt{2\langle M^{(i)}\rangle_t \ln\ln \langle M^{(i)}\rangle_t}}
     = 1\right)=1, \quad i=1,\ldots,d,
 \]
and $\lim_{t\uparrow T} \ln\ln \langle M^{(i)}\rangle_t =\infty$ due to Lemma \ref{lemma_quadratic_variation}.
Similarly to the end of the proof of Lemma \ref{lemma_quadratic_variation} (Case 2),
 for every $\varepsilon'\in(0,a-1/2)$, one can choose a $t_0\in[0,T)$ such that
 \begin{align*}
    \langle M^{(i)}\rangle_t
        & \geq K_1 \int_0^{t_0} \Vert (T-s)^{-A^\top} e_i\Vert^2\,\dd s
             + K_1 \int_{t_0}^t (T-s)^{-2a+2\varepsilon'}\,\dd s \\
        & = K_1 \int_0^{t_0} \Vert (T-s)^{-A^\top} e_i\Vert^2\,\dd s
           +K_1\left( \frac{(T-t)^{-2a+2\varepsilon'+1}}{2a-2\varepsilon'-1}
                      - \frac{(T-t_0)^{-2a+2\varepsilon'+1}}{2a-2\varepsilon'-1}
                 \right)\\
       & = K_1 \frac{(T-t)^{-2a+2\varepsilon'+1}}{2a - 2\varepsilon' - 1} + K_2,  \qquad t\in[t_0,T),
 \end{align*}
with some constants $K_1>0$ and $K_2\in\RR$.
If $K_2\geq 0$, then
 \[
    \langle M^{(i)}\rangle_t
       \geq K_1 \frac{(T-t)^{-2a+2\varepsilon'+1}}{2a - 2\varepsilon' - 1},
       \qquad t\in[t_0,T).
 \]
In case of $K_2<0$ there exists some $t_1\in[t_0,T)$ such that
 \[
    K_2 \geq - K_1 \frac{(T-t)^{-2a+2\varepsilon'+1}}{2(2a - 2\varepsilon' - 1)},
    \quad\text{ for every } t\in[t_1,T),
 \]
 due to $\lim_{t\uparrow T}(T-t)^{-2a+2\varepsilon'+1} = \infty$, and hence we have
 \[
    \langle M^{(i)}\rangle_t
     \geq K_1\frac{(T-t)^{-2a+2\varepsilon'+1}}{2(2a - 2\varepsilon' - 1)},
     \qquad t\in[t_1,T).
 \]
Then, by choosing $\varepsilon'\leq \varepsilon/2$, we get
 \begin{align}\label{help14}
  \begin{split}
    &\Vert (T-t)^{- A + (1/2+\varepsilon)I_d}\Vert_1^{-1} \sqrt{\langle M^{(i)}\rangle_t} \\
    &\qquad \geq \sqrt{\frac{K_1}{2(2a-2\varepsilon'-1)}} \Vert (T-t)^{- A + (1/2+\varepsilon + a - \varepsilon' - 1/2)I_d}\Vert_1^{-1},
        \qquad t\in[t_1,T).
  \end{split}
 \end{align}
Since the real parts of the eigenvalues of the matrix
 $-A +(1/2+\varepsilon + a - \varepsilon' - 1/2)I_d$ are equal to $-a + 1/2 + \varepsilon + a - \varepsilon' - 1/2\geq \varepsilon/2 >0$,
 the right-hand side (and hence the left-hand side) of \eqref{help14} converges to $\infty$ as \ $t\uparrow T$.

Note that the above argument for proving \eqref{help9} works also in the case $a=1/2$ for indices $i\in\{1,\ldots,d\}$
 with $\lim_{t\uparrow T} \langle M^{(i)}\rangle_t=\infty$, since as $t\uparrow T$
 \begin{align*}
  \Vert (T-t)^{-A + (1/2+\varepsilon)I_d}\Vert_1^{-1}
    = \big((T-t)^{1/2+\varepsilon} \sup_{x\in\RR^d \,: \,\Vert x\Vert_1=1} \Vert (T-t)^{-A}x\Vert_1\big)^{-1}
   \to \infty,
 \end{align*}
 where we used that $\lim_{s\downarrow 0}\sup_{x\in R} s^{a+\varepsilon} \Vert s^{-A}x\Vert=0$
 for all $\varepsilon>0$ and for all compact subsets $R$ of $\RR^d$ due Corollary 2.2.5 in Meerschaert and Scheffler \cite{MS}.

{\it Case 3:} $a =\frac12$.
Decompose $\{1,\ldots,d\}=\cI_1\cup \cI_2$ in such a way that the (deterministic) quadratic variation process
 $(\langle M^{(i)}\rangle_t)_{t\in[0,T)}$ is bounded for every $i\in \cI_1$ and $\langle M^{(i)}\rangle_t\to\infty$
 as $t\uparrow T$ for every $i\in \cI_2$.

First we prove \eqref{help8}.
By the proof for the case $a>1/2$, we get for sufficiently large $t$,
 \begin{align*}
  &(T-t)^{-1/2 + \varepsilon} \Vert X_t\Vert\\
  &\quad \leq (T-t)^{-1/2 + \varepsilon} \sum_{i\in\cI_1}
      \sup_{\theta\in\RR^d, \Vert\theta\Vert=1} \left( \Vert \theta^\top (T-t)^A\Vert
      \left\vert \left(\frac{\theta^\top (T-t)^A}{\Vert\theta^\top (T-t)^A\Vert}\right)_i \right\vert\right)
      \vert M^{(i)}_t \vert\\
  &\phantom{\quad\leq} + (T-t)^{-1/2 + \varepsilon} \sum_{i\in\cI_2}
      \sup_{\theta\in\RR^d, \Vert\theta\Vert=1} \left( \Vert \theta^\top (T-t)^A\Vert
      \left\vert \left(\frac{\theta^\top (T-t)^A}{\Vert\theta^\top (T-t)^A\Vert}\right)_i \right\vert\right)\\
  &\phantom{\quad+\leq(T-t)^{-1/2 + \varepsilon} \sum_{i=1}^d }\times
      f(\langle M^{(i)}\rangle_t)
      \frac{\vert M^{(i)}_t \vert}{f(\langle M^{(i)}\rangle_t)},
 \end{align*}
 where the function $f$ is defined in the proof for the case $a>1/2$.
Since, as it was noted, the argument for proving \eqref{help8} in the case $a>1/2$
 works also in the case $a=1/2$ for indices $i\in\cI_2$,
 the second sum on the right-hand side above tends to $0$ as $t\uparrow T$ almost surely.
Next we check that the first sum on the right-hand side above also tends to $0$ as $t\uparrow T$ almost surely.
By Lemma \ref{lln1}, $P\left(\lim_{t\uparrow T}M^{(i)}_t\;\text{exists} \right)=1$ for all $i\in\cI_1$, and hence
 to conclude \eqref{help8}, it is enough to check that for all $\varepsilon>0$, the function
 \[
  (T-t)^{-1/2 + \varepsilon}
    \sup_{\theta\in\RR^d,\, \Vert\theta\Vert=1}  \Vert \theta^\top (T-t)^A\Vert ,
    \qquad t\in[0,T),
 \]
 is bounded.
By \eqref{help11}, for $\varepsilon':=\varepsilon/2$, there exists some $t_0\in[0,T)$ such that
 \[
    \sup_{\theta\in\RR^d,\, \Vert\theta\Vert=1}  \Vert \theta^\top (T-t)^A\Vert
       \leq (T-t)^{1/2-\varepsilon'},
       \qquad t\in[t_0,T),
 \]
 yielding that
 \[
  (T-t)^{-1/2 + \varepsilon}
    \sup_{\theta\in\RR^d, \,\Vert\theta\Vert=1}  \Vert \theta^\top (T-t)^A\Vert
       \leq (T-t)^{\varepsilon - \varepsilon'} = (T-t)^{\varepsilon/2},
       \qquad t\in[t_0,T).
 \]
This implies the desired boundedness.

Next we turn to prove \eqref{help9}.
By the proof for the case $a>1/2$, we get for sufficiently large $t$,
 \begin{align*}
  &(T-t)^{-1/2-\varepsilon}  \Vert X_t\Vert_1\\
  &\quad \geq \Vert (T-t)^{-A + (1/2+\varepsilon)I_d}\Vert_1^{-1}
         \sum_{i\in\cI_1}  \vert M^{(i)} \vert\\
  &\phantom{\quad \geq}+ \Vert (T-t)^{-A + (1/2+\varepsilon)I_d}\Vert_1^{-1}
        \sum_{i\in\cI_2} \sqrt{2\langle M^{(i)}\rangle_t \ln\ln \langle M^{(i)}\rangle_t }
        \frac{\vert M^{(i)}_t \vert}{\sqrt{2\langle M^{(i)}\rangle_t \ln\ln \langle M^{(i)}\rangle_t}}.
 \end{align*}
Since, as it was noted, the argument for proving \eqref{help9} in the case $a>1/2$
 works also in the case $a=1/2$ for indices $i\in\cI_2$,
 the $limsup$ (as $t\uparrow T$) of the second sum on the right-hand side above is $\infty$ almost surely.
By Lemma \ref{lln1}, $P\left(\lim_{t\uparrow T}M^{(i)}_t\;\text{exists} \right)=1$ for all $i\in\cI_1$, and hence
 to conclude \eqref{help9}, it is enough to check that
 \[
   \lim_{t\uparrow T} \Vert (T-t)^{-A + (1/2+\varepsilon)I_d}\Vert_1^{-1} = \infty,
 \]
 which was already shown at the end of the proof for the case $a>1/2$.
\end{proof}

\section{Uniqueness in law of operator scaled Wiener bridges}\label{Section_uniqueness}

For $A,\widetilde A\in\rr^{d\times d}$ and $\Sigma\in\rr^{d\times m},\,\widetilde\Sigma\in\rr^{d\times\widetilde m}$, let the processes $(X_t)_{t\in[0,T)}$
 and $(Y_t)_{t\in[0,T)}$ be given by the SDEs
 \begin{align*}
 \dd X_t & =-\,\frac{1}{T-t}A\,X_t\,\dd t+\Sigma\,\dd B_t,\quad t\in[0,T),\\
 \dd Y_t & =-\,\frac{1}{T-t}\widetilde A\,Y_t\,\dd t+\widetilde\Sigma\,\dd\widetilde B_t,\quad t\in[0,T),
 \end{align*}
 with initial conditions $X_0=0$ and $Y_0=0$, where $(B_t)_{t\geq0}$ and $(\widetilde B_t)_{t\geq0}$ are $m$-, respectively
 $\widetilde m$-dimensional standard Wiener process.
Assume that $(X_t)_{t\in[0,T)}$ and $(Y_t)_{t\in[0,T)}$ generate the same law on the space of real-valued continuous functions defined on $[0,T)$.
Since $(X_t)_{t\in[0,T)}$ and $(Y_t)_{t\in[0,T)}$ are centered Gauss processes, their laws coincide if and only if their covariance functions coincide.
Let $(U(t):=\EE(X_tX_t^\top))_{t\in[0,T)}$ and $(V(t):=\EE(Y_tY_t^\top))_{t\in[0,T)}$ be the corresponding covariance functions.
Then, by Problem 5.6.1 in \cite{KS}, we have
 \begin{align}
 U'(t) & =-\,\frac{1}{T-t}\,AU(t)-U(t)A^\top\frac{1}{T-t}+\Sigma\Sigma^\top,\quad t\in[0,T),\label{uprime}\\
 V'(t) & =-\,\frac{1}{T-t}\,\widetilde AV(t)-V(t)\widetilde A^\top\frac{1}{T-t}
              +\widetilde\Sigma\widetilde\Sigma^\top,\quad t\in[0,T).\label{vprime}
 \end{align}
Since $U(t)=V(t)$ for all $t\in[0,T)$, we get
 \[
  -\,\frac{1}{T-t}\,A U(t)-U(t)A^\top\frac{1}{T-t}+\Sigma\Sigma^\top
  =-\,\frac{1}{T-t}\,\widetilde A U(t)-U(t)\widetilde A^\top\frac{1}{T-t}+\widetilde\Sigma\widetilde\Sigma^\top
 \]
 for $t\in[0,T)$.
Since $U(0)=0\in\RR^{d\times d}$, we have $\Sigma\Sigma^\top=\widetilde\Sigma\widetilde\Sigma^\top$,
 and hence
 \[
 (A-\widetilde A)U(t) = - U(t)(A-\widetilde A)^\top,\quad t\in[0,T).
 \]
Unfortunately, this does not imply that $A=\widetilde A$ in general.
Before we construct counterexamples, we will give the solutions of the $\RR^{d\times d}$-valued differential equations
 \eqref{uprime} and \eqref{vprime} with initial condition $U(0)=0$ and $V(0)=0$, respectively.
By Section 5.6.A in \cite{KS}, one easily calculates that
\begin{align}\label{usolve}
U(t)  =\int_0^t\left(\frac{T-t}{T-s}\right)^A\Sigma\Sigma^\top\left(\frac{T-t}{T-s}\right)^{A^\top}\dd s
\end{align}
 for every $t\in[0,T)$.
Analogously, using also that $\Sigma\Sigma^\top=\widetilde\Sigma\widetilde\Sigma^\top$, we have
 \begin{align}
    V(t) = \int_0^t\left(\frac{T-t}{T-s}\right)^{\widetilde A}\Sigma\Sigma^\top\left(\frac{T-t}{T-s}\right)^{\widetilde A^\top}
           \dd s, \quad t\in[0,T).\label{vsolve}
 \end{align}

Next we give examples for bridges associated to the matrices $A$ and $\Sigma$, and $\widetilde A$ and $\Sigma$,
 respectively, such that their laws on the space of real-valued continuous functions on $[0,T)$ coincide, but $A\ne\widetilde A$.

\begin{example}\label{first}
Let $A\in\RR^{d\times d}$ be a normal matrix, i.e.\ $A A^\top = A^\top A$. Choose $\widetilde A=A^\top$ and let $\Sigma=I_d=\widetilde\Sigma$, then for every $r>0$ we have
\begin{equation}\label{normalid}
 r^A\Sigma\Sigma^\top r^{A^\top}=r^{A+A^\top}=r^{\widetilde A+\widetilde A^\top}=r^{\widetilde A}\Sigma\Sigma^\top r^{\widetilde A^\top}.
 \end{equation}
By \eqref{usolve} and \eqref{vsolve} it follows that $U(t)=V(t)$ for all $t\in[0,T)$. 
Using Theorem \ref{Theorem1}, the bridges associated to the matrices $A$ and $\Sigma$, and $\widetilde A$ and $\widetilde\Sigma$
 coincide, but $A\ne\widetilde A$.
Note also that here the eigenvalues of $A$ and $\widetilde A=A^\top$ coincide.

We further wish to give an example, where the eigenvalues of $A$ and $\widetilde A$ do not coincide, but still $U(t)=V(t)$ holds for all $t\in[0,T)$. Choose the normal matrices
\[
A=\Big(\begin{array}{rc} 1 & 1\\ -1 & 1\end{array}\Big)\quad\text{ and }\quad \widetilde A=I_2
\]
together with $\Sigma=I_2=\widetilde\Sigma$, then due to $A+A^\top=2I_2=\widetilde A+\widetilde A^\top$ again \eqref{normalid} holds for every $r>0$, which yields that $U(t)=V(t)$ for all $t\in[0,T)$ as above. Note that now the eigenvalues $1+i$ and $1-i$ of $A$ do not coincide with those of $\widetilde A=I_2$, but the real parts of the eigenvalues do, including their multiplicity.
\end{example}

To conclude, we formulate a partial result on the uniqueness of the scaling matrix.

\begin{prop}\label{Prop_uniqueness}
Let $A$, $\widetilde A\in\RR^{d\times d}$ and $\Sigma\in\RR^{d\times m}$, $\widetilde\Sigma\in\RR^{d\times \widetilde m}$
 be such that $\RSpec(A)\subseteq(0,1/2)$, $\RSpec(\widetilde A)\subseteq(0,1/2)$ and $\Sigma$, $\widetilde\Sigma$
 have full rank $d$ (and consequently $m\geq d$ and $\widetilde m\geq d$).
If the bridges associated to the matrices $A$ and $\Sigma$, and $\widetilde A$ and $\widetilde\Sigma$ induce the same
 law on the space of real-valued continuous functions on $[0,T)$, then $\RSpec(A)=\RSpec(\widetilde A)$.
\end{prop}

\begin{proof}
The assertion is an immediate consequence of \eqref{help8} and \eqref{help10}.
\end{proof}

\begin{remark}
We conjecture that Proposition \ref{Prop_uniqueness} also holds in the situation $\RSpec(A)\subseteq(0,\infty)$,
 $\RSpec(\widetilde A)\subseteq(0,\infty)$ but we were not able to give a rigorous proof.
\end{remark}

\section*{Acknowledgements}
We are grateful to the referee for several valuable comments that have led to an improvement of the manuscript.

\bibliographystyle{plain}

\end{document}